\documentclass[a4paper,11pt]{amsart}

\usepackage{amsmath, amssymb, amsthm}
\usepackage{graphicx}
\usepackage{subfigure}
\usepackage{color}
\usepackage[hyperindex,breaklinks]{hyperref}
\usepackage{xspace}
\usepackage{framed}
\usepackage{algorithm}
\usepackage[noend]{algpseudocode}

\usepackage[a4paper]{geometry} \newtheorem{theorem}{Theorem}

\newtheorem{assumption}{Assumption}

\newtheorem{lemma}[theorem]{Lemma}
                                                                                                                                                                                                                                                                                                                                                                                                      
\newtheorem{proposition}[theorem]{Proposition}
\newtheorem{remark}{Remark}

\hypersetup{colorlinks = true, urlcolor = blue, linkcolor = blue,
  citecolor = blue }

\usepackage{xspace}
\newcommand{\xmath}[1]{\ensuremath{#1}\xspace}
\newcommand{\E}{\xmath{{E}}}
\renewcommand{\Pr}{\xmath{{P}}}

\newcommand{\gn}{\xmath{\gamma / \phi_n}}
\newcommand{\fin}{\xmath{\phi _n}}
\newcommand{\ega}{\xmath{e^{\gamma }}}

\newcommand{\Zl}{\xmath{Z_{\textnormal{loc}}}}
\newcommand{\Max}{\xmath{\textnormal{Max}}}

\title{Exact and efficient simulation of tail probabilities of
  heavy-tailed infinite series} \date{}
\author[Hult, Juneja and Murthy]{{\large
   H\MakeLowercase{enrik} H\MakeLowercase{ult}} \hspace{20pt}
 {\large S\MakeLowercase{andeep} J\MakeLowercase{uneja}} \hspace{20pt}
 {\large K\MakeLowercase{arthyek} M\MakeLowercase{urthy}}}

\address{Royal Institute of Technology}
\email{hult@kth.se}
\address{Tata Institute of Fundamental Research, Mumbai}
\email{juneja@tifr.res.in}
\address{Columbia University}
\email{karthyek@gmail.com}

\begin{document}
\maketitle
\begin{abstract}
  We develop an efficient simulation algorithm for computing the tail
  probabilities of the infinite series $S = \sum_{n \geq 1} a_n X_n$ when random variables 
  $X_n$ are heavy-tailed.   As
  $S$ is the sum of infinitely many random variables, any simulation
  algorithm that stops after simulating only fixed, finitely many
  random variables is likely to introduce a bias. We overcome this
  challenge by rewriting the tail probability of interest as a sum of
  a random number of telescoping terms, and subsequently developing
  conditional Monte Carlo based low variance simulation estimators for
  each telescoping term. The resulting algorithm is proved to result
  in estimators that a) have no bias, and b) require only a fixed,
  finite number of replications irrespective of how rare the tail
  probability of interest is. Thus, by combining a traditional
  variance reduction technique such as conditional Monte Carlo with
  more recent use of auxiliary randomization to remove bias in a
  multi-level type representation, we develop an efficient  and unbiased simulation
  algorithm for tail probabilities of $S$.
  These have many applications including in 
 analysis of financial time-series
  and stochastic recurrence equations arising in models in actuarial
  risk and population biology.
\end{abstract}

\noindent

\section{Introduction}
\noindent
Given a sequence of regularly varying random variables
$(X_n: n \geq 1)$ and discount factors $(a_n: n \geq 1),$ the
objective of this paper is to design an algorithm that computes tail
probabilities of linear models of form, 
\[S = \sum_{n \ \geq \ 1} a_n X_n.\]
In addition to arising naturally in the study of linear processes and
stochastic recurrence equations, such infinite series are also used in
risk analysis to model instances where, for example, the surplus of an
insurance firm is invested in a risky asset. See \cite{hult2008,
  doi:10.1080/15326340600649029, doi:10.1080/03461230500361943,
  RES97,10.2307/4140398} and references therein for a review of
stochastic models where the infinite series $S$ is a central object of
interest. 

Since exact computation of $P(S > b)$, for a given positive real
number $b,$ is generally not possible, it is common to resort to Monte
Carlo simulations. However, as the object of interest involves
infinitely many random variables, any simulation algorithm that stops
after generating only finitely many random variables is likely to
introduce a bias. In addition, as the parameter $b$ increases, the
event of interest, $\{ S > b\},$ becomes more rare, thus making the
problem harder to estimate within a limited computational budget. The
objective of this paper is to design a Monte Carlo algorithm that
resolves these difficulties. Precisely, we design a family of
simulation estimators $(Z(b): b > 0)$ for estimating probabilities
$P(S > b )$ such that,
\begin{itemize}
 \setlength\itemsep{0em}
\item[1)] the estimators have no bias,
\item[2)] the variance of the family $Z(b)$ is uniformly bounded. 
\end{itemize}
These two properties, in turn, help in guaranteeing that the output
of the Monte Carlo procedure is within a pre-specified relative
precision after only expending an expected computational effort that is
uniformly bounded in $b.$ In other words, the expected computational
effort remains bounded irrespective of the rarity of  the event.

While the study on bias minimization in Monte Carlo simulations
received a huge boost with the introduction of multi-level simulation
(see \cite{giles2008multilevel}), the prospect of eliminating bias
altogether has become possible with the debiasing techniques
employed in \cite{mcleish2011general,rhee2015unbiased} and
\cite{MJB2014}). As we quickly illustrate in Section
\ref{SEC-SIM-METH-CHAP5}, use of a suitably chosen auxiliary random
variable that determines when the algorithm terminates is at
the heart of these new class of algorithms that eliminate bias. For
our simulation problem, this technique enables us to work on modified
`local' problems only involving random variables $a_i, X_i,$ for $i$
not exceeding a random level $N.$ Then the probability law of this
random level $N$ is chosen carefully in order to combine estimators
for these local problems without bias.

Though the use of a suitably chosen auxiliary random variable may
eliminate bias, it is not sufficient to deal with the fact that the
probabilities $P(S > b)$ are small for large values of $b.$
Consequently, a `naive' simulation algorithm will require as many as
$O(1/P(S > b))$ repeated simulation runs to achieve a desired relative
precision (see \cite{juneja2006rare}). Since this is
computationally expensive, we propose new Monte Carlo estimators that
are effective for simulating rare events of our interest. More
precisely, we devise a family of conditional Monte Carlo estimators
$(\Zl(n,b): n \geq 1),$ for a given threshold $b,$ to solve the family
of local problems indexed by $n.$ Here, recall that the $n$-th local
problem is such that it involves only random variables
$(a_i, X_i: i = 1, \ldots, n),$ and hence can be solved in finite
time. Then, assuming that the random variables $X_i$ are regularly
varying, we show that the variance of the local estimators $\Zl(n,b)$ are
sufficiently low, uniformly in $n$ and $b.$

By carefully choosing the law of the auxiliary random variable $N,$ we
combine these local estimators $\Zl(n,b)$ to develop an unbiased
estimator for $P(S > b)$ with bounded coefficient of variation
(relative error). As we show, this ensures that the computational
complexity does not scale with $b,$ even if the probability $P(S > b)$
becomes rare.  In addition, we verify that the expected termination
time of the simulation algorithm is finite, thereby guaranteeing that
the estimation can performed with a computational effort that is
uniformly bounded in $b.$ 

As estimation of tail probabilities of heavy-tailed sums has been
known to be more challenging than their light-tailed counterparts (see
\cite{ABH00,MR2311409}), simulation algorithms using a variety of
techniques (such as conditional Monte Carlo (see \cite{AK06}),
importance sampling (see
\cite{ABH00,Juneja:2002:SHT:566392.566394,DBLP:journals/questa/Juneja07,Dupuis:2007:ISS:1243991.1243995,MR2434174,MR2488534,Blanchet20122994,MJB2014},
splitting \cite{Blanchet:2013:ESR:2675983.2676077}, Markov chain Monte
Carlo \cite{gudmundsson2014}, and cross-entropy method \cite{BlSh13}
have been developed after intense research over previous decade. In
particular, \cite{Blanchet:2013:RSS:2556945.2517451} develops an
importance sampling algorithm for simulation of tail probabilities of
the stochastic recurrence equations of the form
\begin{align*}
  X_{n+1} = A_{n + 1}X_n + B_{n+1}, \quad X_0 = 0,
\end{align*}
for large values of $n$. As noted by them, this has  applications in a variety of
settings ranging from financial time series, actuarial risk and
population biology (see
\cite{Kesten1973,goldie1991,roitershtein2007,Basrak200295,Lewontin01041969,tuljapurkar1990delayed}
and references therein). As $n \rightarrow \infty,$ the distribution
of $X_n$ corresponds to the stationary distribution of the Markov
chain modeled by $X_n,$ and can be representated as an  infinite series (see \cite{hult2008}). While importance
sampling remains the most studied technique for simulation of such
tail probabilities, conditional Monte Carlo estimators devised by
Asmussen and Kroese \cite{AK06} have been shown to offer superior
numerical accuracy (see Section 3.5 in \cite{Murthy_Dissertation}). Leveraging this, we design intuitive, easy-to-use Asmussen-Kroese type
conditional Monte Carlo estimators to solve the local problems
mentioned earlier\footnote{The proposed estimators and their variance analysis
comprise Chapter 5 in the PhD Dissertation \cite{Murthy_Dissertation}
of one of the authors}. As the sampling techniques involved for
simulating rare events in light-tailed sums are drastically different
(see, for example, \cite{MR1053850,juneja2006rare,MR2743897}), we note
that a similar study for estimation of tail probabilities of
light-tailed infinite sums as an interesting future research
direction.

The paper is organized as follows: After describing the problem of
interest precisely in Section \ref{Sec-Notn-Prob-St}, we develop the
simulation methodology in Section
\ref{SEC-SIM-METH-CHAP5}. 
A detailed variance analysis that characterizes the
computational complexity of the family of local estimators and the
overall estimators introduced in Section \ref{SEC-SIM-METH-CHAP5} is
presented in Section \ref{SEC-VAR-ANAL-CHAP5}. 
A numerical
example that reaffirms the theoretical efficiency results of the paper
is presented in Section \ref{SEC-NUM-EG-CHAP3}. Technical proofs that
are not central to the variance analysis in Section
\ref{SEC-VAR-ANAL-CHAP5} are presented in the appendix.

\section{Notations and problem statement}
\label{Sec-Notn-Prob-St}
\noindent
To precisely introduce the problem, let $X$ be a zero mean random
variable satisfying the following condition:
\begin{assumption}
  The distribution function of $X,$ denoted by $F(\cdot),$ is such that
  the tail probabilities $\bar{F}(x) := 1 - F(x) = x^{-\alpha} L(x)$
  for some slowly varying function $L(\cdot)$ and $\alpha > 2.$
  \label{REG-VAR-TAIL-ASSUMP} 
\end{assumption}
\noindent Here, the slowly varying function $L(\cdot)$ stands for any
function that satisfies $L(tx)/L(x) \rightarrow 1,$ for every $t > 0,$
as $x \rightarrow \infty.$ When $L(\cdot) = c$ for some positive
constant $c,$ we obtain the special case of Pareto (or) power-law
distributions. Other common examples of slowly varying functions
include logarithmically decaying/growing functions such as
$\log x, \log \log x, 1/ \log x,$ etc. The following
property, 
commonly referred as Potter's bounds, confirms that
regularly varying tail probabilities $\bar{F}(\cdot)$ are essentially
polynomially decaying: there exists a $t_\delta > 0$ such that for all
$t$ and $v$ satisfying $t \geq t_\delta$ and $vt \geq t_\delta$,
\begin{equation}
  (1-\delta) \min \{v^{-\alpha + \delta}, v^{-\alpha-\delta}\} \leq \frac{\bar{F}(vx)}{\bar{F}(x)}
  \leq (1+\delta)\max \{v^{-\alpha + \delta}, v^{-\alpha-\delta}\}
  \label{LONG-TAIL-EXT}.
\end{equation}
See, for example, Chapter VIII of \cite{feller1971introduction} or
Chapter 1 of \cite{MR2424161}, for a proof of \eqref{LONG-TAIL-EXT},
and other important properties of regularly varying distributions.

Let $(X_n: n \geq 1)$ be a sequence of i.i.d. copies of $X.$ Our aim
is to efficiently estimate the tail probabilities of
\[S := \sum_n a_n X_n,\] where $(a_n : n \geq 1)$ satisfies the
following condition: \begin{assumption} The sequence $(a_n: n \geq 1)$
  is such that $a_n$ lies in the interval $(0,1)$ for every $n$ and
  $\sum_n na_n < \infty.$
\label{COEFF-ASSUMP-1}
\end{assumption}
The random variable $S$ is proper because $\sum_n a_n^2 < \infty$
(follows from Kolmogorov's three-series theorem).  The assumptions
that $X$ has zero mean and $a_n < 1$ have been made just for the ease
of exposition. If $X$ has non-zero mean or if $a_n > 1$ for any $n,$
then the corresponding problem of estimating $ \Pr \{S > b\}$ can be
translated to a problem instance satisfying Assumptions
\ref{REG-VAR-TAIL-ASSUMP} and \ref{COEFF-ASSUMP-1} by letting
$\tilde{a}_n := a_n/\sup_n a_n$ and by instead simulating the right
hand side of the equation below:
\[ \Pr \left\{ \sum_n a_n X_n > b \right\} = \Pr \left\{ \sum_n
  \tilde{a}_n (X_n - \E X) > \frac{b - \left( \sum_n a_n \right) \E
    X}{\sup_n a_n} \right\}.\]
Here note that $\sup_n a_n$ exists because we require
$\sum_n a_n < \infty.$ 




\section{Simulation Methodology}
\label{SEC-SIM-METH-CHAP5}
Given $b > 0,$ we aim to estimate   $\Pr \{ S > b \}$ via simulation.  If $S$ is a sum of, say, for
example, $k$ i.i.d. random variables $X_1,\ldots,X_k$, then one can
simply generate an i.i.d. realization of $X_1,\ldots,X_k$ and check
whether their sum is larger than $b$ or not. However, the countably
infinite number of random variables involved in the definition of $S$
makes the task of obtaining a sample of $S$ via its increments, at
least at a preliminary look, appear computationally infeasible. To
overcome this difficulty, we introduce an auxiliary random variable
$N$ and re-express the probability $\Pr\{ S > b \}$ below in
\eqref{RE-EXP-II} in a form that gives computational tractability: Let
\[ S_0 := 0 \text{ and } S_n := \sum_{i=1}^n a_i X_i \text{ for } n
\geq 1. \] Further, let $p_n := \Pr\{ N =n\}$ be positive for every $n
\geq 1.$ Then,
\begin{align}
  \Pr \{ S > b \} &= \lim_n \Pr \{ S_n > b\} \nonumber\\
  &= \sum_{n \geq 1} p_n \frac{\Pr \left\{ S_n > b \right\} - \Pr
    \left\{ S_{n-1} > b \right\}}{p_n} \nonumber\\
  &= \E \left[ \frac{\Pr \left\{ S_{_N} > b \ |\ N \right\} - \Pr
      \left\{ S_{_{N-1}} > b \ |\ N \right\} }{p_{_N}} \right]
  \label{RE-EXP-II}
\end{align}
In Section \ref{SEC-LOC-EST}, we aim to develop unbiased estimators
$(\Zl(n,b): n \geq 1, b> 0)$ satisfying the following desirable
properties:
\begin{itemize}
\item[(1)] The expectation of $\Zl(n,b)$ is $\Pr\{S_n > b\} - \Pr \{
  S_{n-1} > b\}$ for every $n$ and $b.$
\item[(2)] The computational effort required to generate a realization
  of $\Zl(n,b)$ is bounded from above by $C n,$ for some constant $C >
  0,$ uniformly for all $b.$
\item[(3)] The estimators $\Zl(n,b)$ have low variance, uniformly in
  $n$ and $b.$
\end{itemize}
Now, in a simulation run, if the realized value of $N$ is $n,$ we
generate an independent realization of estimator $\Zl(n,b)$ and use
\begin{align*}
  Z(b) := \frac{\Zl(N,b)}{p_{_N}}
\end{align*}
as an estimator for $\Pr \{ S > b\}.$ The fact that $Z(b)$ yields
estimates of $\Pr \{ S > b\}$ without any bias follows from
\eqref{RE-EXP-II}. Thus by introducing an auxiliary random variable
$N,$ in every simulation run, we are faced with the task of generating
only finitely many random variables, as opposed to the naive approach
which requires generation of countably infinite random variables. The
random variables $\Zl (n,b)$ which are instrumental in estimating the
tail probabilities of $S$ will be referred hereafter as `local'
estimators.

\subsection{Local estimators}
\label{SEC-LOC-EST}
As mentioned before, in this section, we present estimators for
quantities \[\left(\Pr\{S_{n} > b\} - \Pr\{S_{n-1} > b\}: n \geq
  1\right)\] that have low variance, uniformly in $n,$ as $b \rightarrow
\infty.$ These form building blocks to serve our initial aim of
estimating the tail probabilities of $S.$ It is well-known that the
sum of heavy-tailed random variables attain a large value typically
because one of the increments (and hence the maximum of the
increments) attain a large value. Therefore, we focus our attention on
identifying the maximum of the increments
\[ M_n := \max\{ a_iX_i: 1 \leq i \leq n\}\]
in a manner that is reflective of the way in which the rare event
under consideration happens. For this, we partition the sample space
based on which of the $n$ increments $\{ a_1X_1,\ldots,a_nX_n\}$ is
the maximum. Let $\Max_n$ denote the index of the increment $a_iX_i$
that equals the maximum $M_n.$ In case of many increments having the
same value as the maximum, we take the largest (index) of them to be
$\Max_n.$ That is,
\[ \Max_n := \max\{ \textnormal{argmax}\{a_iX_i: 1 \leq i \leq n\} \}.\] 
See that the quantity $\Pr\{S_{n} > b\} - \Pr\{S_{n-1} > b\}$ can be
alternatively expressed as
\begin{align}
  \label{PARTITION-ON-MAX}
 \Pr\{S_{n} > b\} - \Pr\{S_{n-1} > b\} = p_1(n,b) + p_2(n,b)  
\end{align}
 where
\begin{align*}
  p_1(n,b) &= \Pr \left\{S_{n} > b, \Max_n = n \right\} - \Pr
  \left\{S_{n-1} > b, \Max_n = n\right\}
  \text{ and }\\
  p_2(n,b) &= \Pr \left\{S_{n} > b, \Max_n \neq n \right\} - \Pr
  \left\{S_{n-1} > b, \Max_n \neq n\right\}.
\end{align*}
We develop alternative representations for quantities $p_1(n,b)$ and
$p_2(n,b)$ and use them to separately estimate $p_1(n,b)$ and
$p_2(n,b)$ in the following sections.

\subsubsection{Estimator for $p_1(n,b)$}
Observe that $\Pr \left\{S_{n-1} > b, S_{n} \leq b, \Max_n = n
\right\} = 0$ because whenever $S_n \leq b$ and $S_{n-1} > b,$ it is
necessary that $X_n$ be negative, and in which case $S_n$ also needs
to be negative (since $M_n = a_nX_n$). Therefore,
\begin{align*}
  p_1(n,b) &= \Pr \left\{S_{n} > b, S_{n-1} \leq b, \Max_n = n
  \right\} - \Pr \left\{S_{n-1} > b, S_{n} \leq b, \Max_n = n
  \right\}\\
  &=\Pr \left\{S_{n} > b, S_{n-1} \leq b, \Max_n = n \right\}.
\end{align*}
Further,
\begin{align*}
  \Pr \left \{ S_n > b, \Max_n = n \ \left| \frac{}{}
    \right. X_1,\ldots,X_{n-1}\right\} & = \Pr \left\{ a_n X_n > b -
    S_{n-1}, a_n X_n > M_{n-1} \ \left| \frac{}{}
    \right. X_1,\ldots,X_{n-1}\right\} \nonumber\\
  & = \bar{F}\left( \frac{1}{a_{_n}} \left(\left( b - S_{n-1}\right)
      \vee M_{n-1}\right) \right). 
\end{align*}
Therefore, it is immediate that 
\begin{align*}
  \E \left[ \bar{F}\left( \frac{1}{a_{_n}} \left(\left( b -
          S_{n-1}\right) \vee M_{n-1}\right) \right)
    \mathbb{I}(S_{n-1} \leq b)\right] 
  &= \Pr \left\{ S_n > b, S_{n-1} \leq b, \Max_n = n\right\}.
\end{align*}
If we let
\[ Z_1(n,b) := \bar{F}\left( \frac{1}{a_{_n}} \left(\left( b -
      S_{n-1}\right) \vee M_{n-1}\right) \right) \mathbb{I}(S_{n-1}
\leq b), \] then it follows from the above discussion that
$\E\left[Z_1(n,b)\right]$ equals $p_1(n,b).$ We note this observation
below as Lemma \ref{LEM-UNB-Z1}.
\begin{lemma} 
  For every $n > 1$ and $b > 0,$ $\E \left[ Z_1(n,b) \right] =
  p_1(n,b).$
\label{LEM-UNB-Z1}
\end{lemma}
In a simulation run, one can generate samples of $X_1,\ldots,X_{n-1}$
simply from the distribution $F(\cdot)$ and plug it in the expression
of $Z_1(n,b)$ to arrive at an unbiased estimator for $p_1(n,b).$ Since
$Z_1(n,b)$ is just the probability that the event of interest
$\{ S_n > b, S_{n-1} \leq b, \Max_n = n\}$ happens conditional on the
observed values of $X_1,\ldots,X_{n-1},$ $Z_1(n,b)$ is said to belong
to a family of estimators called conditional Monte Carlo estimators
(see, for example, \cite{MR2331321}). Estimators of the form
$Z_1(n,b),$ also referred to as Asmussen-Kroese estimators, are shown
to be extremely effective in the simulation of tail probabilities of
sums of fixed number of heavy-tailed random variables in \cite{AK06}.

\subsubsection{Estimator for $p_2(n,b)$}
Similar to $p_1(n,b),$ one can develop conditional Monte Carlo
estimators for the simulation of $p_2(n,b)$ as well. To accomplish
this, we need  more notation: For any $j \leq n,$ let
\begin{align*}
  S_n^{(-j)} := \sum_{i=1, i \neq j}^n a_i X_i \text{ and }
  M_n^{(-j)} := \max_{i \leq n, i \neq j} a_i X_i.
\end{align*}
Further, for any $n > 1,$ let $(q(j,n): 0 < j < n)$ be a probability
mass function that assigns positive probability to every integer in
$\{1,\ldots,n-1\}.$ Let $J_n$ be an auxiliary random variable which
takes values in $\{1,\ldots,n-1\}$ such that $\Pr \{ J_n = j \} =
q(j,n).$ Aided with this notation, define the estimator for $p_2(n,b)$
as
\begin{align*}
  Z_2(n,b) &:= \frac{Z_{2,1}(n,b) - Z_{2,2}(n,b)}{q(J_n,n)},
\end{align*}
where
\begin{align*}
  Z_{2,1} (n,b) &:= \bar{F} \left( \frac{1}{a_{_{J_n}}} \left( \left(
        b -S_{n}^{(-J_n)} \right) \vee
      M_n^{(-J_n)}\right) \right) \text{ and }\\
  Z_{2,2} (n,b) &:= \bar{F} \left( \frac{1}{a_{_{J_n}}} \left( \left(
        b -S_{n-1}^{(-J_n)} \right) \vee M_n^{(-J_n)} \right) \right).
\end{align*}
Lemma \ref{LEM-UNB-Z2} below verifies that $Z_2(n,b)$ is an unbiased
estimator for $p_2(n,b).$

\begin{lemma} 
  For every $n > 1$ and $b > 0,$ $\E \left[ Z_2(n,b) \right] =
  p_2(n,b).$
  \label{LEM-UNB-Z2}
\end{lemma}
\begin{proof}
  For any $n$ and $j < n,$ observe that
  \begin{align}
    \Pr \left \{ S_n > b, \Max_n = j \ \left| \frac{}{}
      \right. S_n^{(-j)}, M_n^{(-j)} \right\} & = \Pr \left\{ a_j X_j
      > b - S_n^{(-j)}, a_j X_j > M_n^{(-j)} \ \left| \frac{}{}
      \right. S_n^{(-j)}, M_n^{(-j)}
    \right\} \nonumber\\
    & = \bar{F}\left( \frac{1}{a_{_j}} \left(\left( b -
          S_n^{(-j)}\right) \vee M_n^{(-j)}\right)
    \right), \label{INTER-PROB-TO-EST-1}
  \end{align}
  and similarly,
  \begin{align}
    \Pr \left \{ S_{n-1} > b, \Max_n = j \ \left| \frac{}{}
      \right. S_{n-1}^{(-j)}, M_n^{(-j)} \right\} &= \bar{F}\left(
      \frac{1}{a_{_j}} \left(\left( b - S_{n-1}^{(-j)}\right) \vee
        M_n^{(-j)}\right) \right). \label{INTER-PROB-TO-EST-2}
  \end{align}
  Recall that $J_n$ takes values only in $\{1,\ldots, n-1\}.$
  Therefore, it follows from the definition of $Z_{2,1}(n,b)$ and
  $Z_{2,2}(n,b)$ that
  \begin{align*}
    Z_{2,1}(n,b) &= \Pr \left\{ S_n > b, \Max_n = J_n \ \left|
        \frac{}{} \right. S_n^{(-J_n)}, M_n^{(-J_n)}, J_n \right\}
    \text{ and }\\
    Z_{2,2}(n,b) &= \Pr \left\{ S_{n-1} > b, \Max_n = J_n \ \left|
        \frac{}{} \right. S_n^{(-J_n)}, M_n^{(-J_n)}, J_n \right\}.
  \end{align*}
  Then it is immediate that
  \begin{align*}
    \E \left[ Z_{2,1}(n,b) \ |\ J_n \right] &= \Pr \left \{ S_n > b,
      \Max_n = J_n\ \left| \frac{}{} \right. J_n \right\} \text{ and
    }\\
    \E \left[ Z_{2,2}(n,b) \ |\ J_n \right] &= \Pr \left \{ S_{n-1} >
      b, \Max_n = J_n\ \left| \frac{}{} \right. J_n \right\}.
  \end{align*}
  Since $\Pr \{ J_n = j\} = q(j,n),$ it follows that
  \begin{align}
    \E \left[ \frac{Z_{2,1}(n,b)}{q(J_n,n)} \right] 
    &= \sum_{j=1}^{n-1} \Pr \{ J_n = j\} \E \left[
      \frac{Z_{2,1}(n,b)}{q(J_n,n )} \ \left| \frac{}{} \right.  J_n = j\right]\\
    &=   \sum_{j=1}^{n-1} q(j,n) \frac{\E \left[ Z_{2,1}(n,b) \ | \ 
      J_n = j\right]}{q(j,n)} \nonumber\\
    &= \sum_{j=1}^{n-1} \Pr \left \{ S_n > b, \Max_n = j \right\}
      \nonumber\\
    &= \Pr \left \{ S_n > b, \Max_n \neq n \right\}.
    \label{INTER-Z21}
  \end{align}
  Similarly one can derive that
  \begin{align}
    \E \left[ \frac{Z_{2,2}(n,b)}{q(J_n,n)} \right] &= \Pr \left \{
      S_{n-1} > b, \Max_n \neq n \right\} \label{INTER-Z22}
  \end{align}
  Since $Z_2(n,b) = (Z_{2,1}(n,b)-Z_{2,2}(n,b))/q(J_n,n),$ it is
  immediate from \eqref{INTER-Z21} and \eqref{INTER-Z22} that $\E
  [Z_2(n,b)] = p_2(n,b).$
\end{proof}
\noindent To summarize the simulation procedure, we present Algorithm
\ref{ALGO-SIM-LOC} here, which returns a realization of
\[\Zl(n,b) := Z_1(n,b) + Z_2(n,b)\]
for given values of $n$ and $b.$ It follows from Lemmas
\ref{LEM-UNB-Z1} and \ref{LEM-UNB-Z2} that $\Zl(n,b)$ is indeed an
unbiased estimator for the quantity $\Pr \{S_n > b\} - \Pr \{ S_{n-1}
> b\}.$
\begin{algorithm}[h!]
  \caption{Given $n$ and $b,$ the aim is to efficiently simulate $\Pr
    \{S_n > b\}-\Pr \{ S_{n-1} > b\}$ }
    \begin{algorithmic}
      \\
      \Procedure{LocalSimulation}{$n,b$} 
 
      \State Let $Z_1(n,b) = \textsc{Estimator1}(n,b) \text{ and
      }Z_2(n,b) = \textsc{Estimator2}(n,b)$

      
      \State Return $\Zl(n,b) = Z_1 (n,b) + Z_2 (n,b)$\\

      \EndProcedure

      \Procedure{Estimator1}{$n,b$} \State Initialize $Z_1(n,b) = 0$
      \State Simulate a realization of $(X_i: 1 \leq i \leq n-1)$
      independently from the distribution $F(\cdot)$

      \State Let $S_{n-1} = \sum_{i=1}^{n-1} a_i X_i \text{ and }
      M_{n-1} = \max \{ a_i X_i: 1 \leq i \leq n-1\}$

      \If{ $S_{n-1} \leq b$ } \State Let \[ Z_1(n,b) = \bar{F} \left(
        \frac{1}{a_n} \left(\left( b -S_{n-1} \right) \vee M_{n-1}
        \right)\right)\]
      \EndIf 
      \State Return $Z_1(n,b)$
      \\
      \EndProcedure

      \Procedure{Estimator2}{$n,b$} \State Generate a sample of $J_n$
      such that for $j = 1,\ldots,n-1, \Pr \{ J_n = j\} = q(j,n) :=
      a_j/\sum_{i=1}^{n-1}a_i$

      \State For $1 \leq i \leq n, i \neq J_n$ simulate $X_i$
      independently from the distribution $F(\cdot)$

      \State Let $S_n^{(-J_n)} = \sum_{i=1, i \neq J_n}^n a_i X_i,
      \ M_n^{(-J_n)} = \max\{a_i X_i: i \leq n, i \neq J_n\},$
      \begin{align*}
        Z_{2,1} (n,b) &= \bar{F} \left( \frac{1}{a_{_{J_n}}} \left(
            \left( b -S_{n}^{(-J_n)} \right) \vee
            M_n^{(-J_n)}\right) \right),\\
        Z_{2,2} (n,b) &= \bar{F} \left( \frac{1}{a_{_{J_n}}} \left(
            \left( b -S_{n-1}^{(-J_n)} \right) \vee
            M_n^{(-J_n)} \right) \right) \text{ and }\\
        Z_2(n,b) &= \frac{Z_{2,1}(n,b) - Z_{2,2}(n,b)}{q(J_n,n)}.
      \end{align*} 
      \State Return $Z_2(n,b)$
      \\
      \EndProcedure
    \end{algorithmic}
\label{ALGO-SIM-LOC} 
\end{algorithm}

\subsection{Simulation of $\Pr \{ S > b\}$}
\label{OVRALL-ALG-DETAILS}
We use an auxiliary random variable $N$ to estimate the tail
probabilities of the infinite series $S = \sum_n a_n X_n.$
Recall that \textsc{LocalSimulation$(n,b)$} is
a simulation procedure introduced in Algorithm \ref{ALGO-SIM-LOC} in
Section \ref{SEC-LOC-EST}, which for given values of $n$ and $b,$
returns realizations of random variable $\Zl(n,b)$ that has $\Pr\{S_n
> b\} - \Pr\{S_{n-1} > b\}$ as its expectation. Given $b > 0,$ we
present below Algorithm \ref{ALGO-OVRALL} that makes a call to
\textsc{LocalSimulation} procedure of Algorithm \ref{ALGO-SIM-LOC} and
returns
\[Z(b) := \frac{\Zl(N,b)}{p_{_N}}\] which is the estimator we propose
for computing the probability $\Pr \{ S > b\}.$

\begin{algorithm}[h]
  \caption{ Given $b > 0,$ the aim is to efficiently simulate $\Pr \{S
    > b\}$ }
    \begin{algorithmic}
      \\
      \State Generate a sample of $N$ such that $\Pr \{ N = n\} = p_n,
      \text{ for } n \geq 1$
   
      \State Let $\Zl(N,b) = \textsc{LocalSimulation}(N,b)$

      \State Let \[ Z(b) = \frac{\Zl (N,b)}{p_{_N}} \]

      \State Return $Z(b)$
    \end{algorithmic}
\label{ALGO-OVRALL} 
\end{algorithm}

\begin{theorem}
  The estimators $(Z(b): b > 0)$ are unbiased: that is, for every $b >
  0,$ 
  \[\E \left[Z(b)\right] = \Pr \{ S > b\}.\] 
  \label{THM-UNB-OVRALL}
\end{theorem}
\begin{proof}
  Since $\E \left[ \Zl (n,b)\right] = \Pr\{S_n > b\} - \Pr\{S_{n-1} >
  b\}$ for every $n$ and $b,$
  \begin{align*}
    \E \left[Z(b) \right] &= \E \left[ \E \left[
        \frac{\Zl(N,b)}{p_{_N}}
        \left|\frac{}{}\right. N \right]\right]\\
    &= \E \left[ \frac{\Pr \left\{S_{_N} > b \ |\ N \right\}-\Pr
        \left\{S_{_{N-1}}>b \ |\ N \right\}}{p_{_N}} \right]\\
    &= \sum_n\Pr\{N=n\} \frac{\Pr \left\{S_{n} > b \right\}-\Pr
      \left\{S_{n-1}>b\right\}}{p_{n}}.
  \end{align*}
  Since $\Pr\{ N = n \} = p_n,$ it is immediate that,
  \begin{align*}
    \E \left[Z(b) \right] = \sum_n\left[\Pr \left\{S_{n} > b
      \right\}-\Pr \left\{S_{n-1}>b\right\}\right] = \lim_n \Pr \{
    S_{n} > b \}.
  \end{align*}
  Since $S_n \rightarrow S$ almost surely, as $n \rightarrow \infty,$
  $\lim_n \Pr \{ S_n > b \}$ equals $\Pr \{ S > b \}.$ Thus, we have
  that the estimators $Z(b)$ are unbiased.
\end{proof}

Theorem \ref{THM-UNB-OVRALL} above re-emphasizes the fact that $Z(b)$
returned by Algorithm \ref{ALGO-OVRALL} is unbiased in the estimation
of $\Pr \{S > b\}$ for every choice of $(p_n: n \geq 1)$ satisfying
$p_n > 0$ and $\sum_n p_n = 1.$ However, for our simulation procedure, 
we take
\begin{equation}
  p_n := c_b \left( a_n^\alpha + \frac{a_n}{b^r}\right),
  \label{PROB-CHOICE-OUTER-RANDOMIZATION}
\end{equation}
for some $r \geq 1.$ As one can infer from the variance analysis in
Section \ref{SEC-VAR-ANAL-CHAP5}, the choice of $(p_n: n \geq 1)$ as
in \eqref{PROB-CHOICE-OUTER-RANDOMIZATION} is the smallest choice that
makes the ratio $\E\left[\Zl^2(n,b)\right]/p_n^2$ uniformly bounded by
a positive constant that is not dependent on $n.$

\section{Analysis of Variance of $Z(b)$}
\label{SEC-VAR-ANAL-CHAP5}
The aim of this section is to prove the following theorem when
Assumptions \ref{REG-VAR-TAIL-ASSUMP} and \ref{COEFF-ASSUMP-1} are in
force: 
\begin{theorem}
  For the choice of probabilities $(p_n: n \geq 1)$ as in
  (\ref{PROB-CHOICE-OUTER-RANDOMIZATION}), if $r$ is taken larger than
  1, the family of estimators $(Z(b): b > 0)$ returned by Algorithm
  \ref{ALGO-OVRALL} has vanishing relative error, asymptotically, as
  $b \rightarrow \infty.$ In other words,
  \begin{align*}
    \lim_{b \rightarrow \infty} \frac{\E \left[ Z^2(b) \right]}{\Pr \{
      S > b \}^2} = 1.
  \end{align*}
  \label{THM-EFF-OVRALL}
\end{theorem}
\noindent To prove that the estimators $Z(b)$ have low variance
asymptotically as in the statement of Theorem \ref{THM-EFF-OVRALL}, we
need to establish that $\E [\Zl^2(n,b)]$ is comparable to that of
$p_n^2\bar{F}^2(b),$ which is challenging because proving such a
proposition will have to establish that $\E[\Zl^2(n,b)]$ is low with
respect to two rarity parameters $n$ and $b.$ We accomplish this in
the following section.

\subsection{Uniform bounds on variance of local estimators}
To obtain bounds on variance of estimators $\Zl (n,b),$ we separately
analyse the second moments of $Z_1(n,b)$ and $Z_2(n,b)$ (defined in
Algorithm \ref{ALGO-SIM-LOC}) below. Proposition \ref{PROP-RES-PROB}
which is stated below and proved in the appendix will be useful in the
analysis.

\begin{proposition}
  Under Assumptions \ref{REG-VAR-TAIL-ASSUMP} and
  \ref{COEFF-ASSUMP-1}, 
  \begin{align*}
    \Pr \left\{ S_n^{(-j)} > b , \ M_n^{(-j)} \leq \frac{b}{k}\right\}
    &\leq \exp \left(k+o(1)\right) \left( \frac{\sum_i a_i^\alpha}{k}
      \bar{F}\left(\frac{b}{k}\right)\right)^k, \text{ as } b \rightarrow
    \infty
  \end{align*}
  uniformly in $n,$ for every $j \leq n$ and $k > 1.$
  \label{PROP-RES-PROB}
\end{proposition}

\begin{remark}
  \textnormal{For large values of $b,$ Proposition \ref{PROP-RES-PROB} roughly
  captures the idea that when the maximum of the increments are
  constrained, for example, to be smaller than $b/2,$ the likely way
  for a heavy-tailed sum to become larger than $b$ is by having two
  large increments roughly of size $b/2.$ Though $k$ being an integer
  helps in understanding the upper bound in Proposition
  \ref{PROP-RES-PROB} in terms of the number of jumps, one can check
  from the proof of Proposition \ref{PROP-RES-PROB} that
  the upper bound holds true for $k$ being any real number larger than
  1.}
\end{remark}

\subsubsection{Analysis of $Z_1(n,b)$}
Recall that 
\[ Z_1(n,b) := \bar{F} \left( \frac{1}{a_n} \left(\left( b -S_{n-1}
    \right) \vee M_{n-1} \right)\right).\]
To upper bound second moment of $Z_1(n,b),$ we consider the following
two quantities:
\begin{align*}
  I_1(n,b) &:= \E \left[ Z_1^2(n,b); (b-S_{n-1}) \vee M_{n-1}
    \geq \gamma b \right] \text{ and }\\
  I_2(n,b) &:= \E \left[ Z_1^2(n,b); (b-S_{n-1}) \vee M_{n-1} < \gamma
    b \right]
\end{align*}
for $\gamma \in (0,1).$
\begin{lemma}
  Under Assumptions \ref{REG-VAR-TAIL-ASSUMP} and
  \ref{COEFF-ASSUMP-1},
  \[ \varlimsup_{b \rightarrow \infty} \sup_{n > 1} \frac{I_1(n,b)}{
    (a_n^{\alpha - \delta} \bar{F}(b))^2} \leq (1+\delta)^2\] for
  every $\delta > 0$ and $\gamma \in (0,1).$
  \label{LEM-I1}
\end{lemma}
\begin{proof}
  From the definition of $Z_1(n,b),$ it is immediate that
  \begin{align*}
    I_1(n,b) \leq \bar{F}^2(b) \E \left[
      \frac{\bar{F}^2\left(\frac{b}{a_n}\left(\left(1-\frac{S_{n-1}}{b}\right)
            \vee \gamma \right)\right)}{\bar{F}^2(b)}\right].
  \end{align*}
  Since $\bar{F}(x) = x^{-\alpha + o(1)},$ given $\delta > 0,$ for $b$
  large enough,  because of \eqref{LONG-TAIL-EXT},
  we have that for
  every $n,$
  \begin{align*}
    \frac{\bar{F}\left(\frac{b}{a_n}\left(\left(1-\frac{S_{n-1}}{b}\right)
          \vee \gamma \right)\right)}{\bar{F}(b)} \leq (1+\delta)
    a_n^{\alpha-\delta} h\left( \frac{S_{n-1}}{b} \right),
  \end{align*}
  where $h(x) = \left((1-x) \vee \gamma \right)^{-(\alpha+\delta)}.$
  Therefore,
  \begin{align*}
    \sup_{n \geq 1}
    \frac{I_1(n,b)}{\left(a_n^{\alpha-\delta}\bar{F}(b)\right)^2} \leq
    (1+\delta)^2 \sup_{n \geq 1} \E \left[ h^2 \left(
        \frac{S_{n-1}}{b} \right)\right].
  \end{align*}
  Since $h(\cdot)$ is a non-decreasing function, it is immediate that
  \begin{align*}
    \sup_{n \geq 1}
    \frac{I_1(n,b)}{\left(a_n^{\alpha-\delta}\bar{F}(b)\right)^2} \leq
    (1+\delta)^2 \E \left[ h^2 \left( \frac{\sum_n a_n X_n^+ }{b}
      \right)\right],
  \end{align*}
  where $x^+ := \max\{x,0\}$ for $x \in \mathbb{R}.$ The following
  observations are in order:
  \begin{itemize}
  \item[1)] $h(\cdot)$ is bounded
  \item[2)] The random variable $\sum_n a_n X_n^+$ is proper (this is
    because $\sum_n a_n < \infty$ and hence a consequence of
    Kolmogorov's three-series theorem). Therefore,
    $b^{-1}\sum_na_nX_n^+ \rightarrow 0$ almost surely, as $b \rightarrow
    \infty.$
  \end{itemize}
  Then because of bounded convergence,
  \begin{align*}
    \E \left[ h^2\left( \frac{\sum_na_nX_n^+}{b}\right)\right]
    \rightarrow 1, \text{ as } b \rightarrow \infty.
  \end{align*}
  Thus, for every $\delta > 0,$ we have that
  \begin{align*}
    \varlimsup_{b \rightarrow \infty} \sup_{n \geq 1}
    \frac{I_1(n,b)}{\left(a_n^{\alpha-\delta}\bar{F}(b)\right)^2} \leq
    (1+\delta)^2.
  \end{align*}
\end{proof}

\begin{lemma}
  Under Assumptions \ref{REG-VAR-TAIL-ASSUMP} and
  \ref{COEFF-ASSUMP-1}, there exists $\gamma$ in $(0,1)$ such that
  \[ \varlimsup_{b \rightarrow \infty} \sup_{n > 1} \frac{\E \left[
      I_2(n,b)\right]}{ (p_n \bar{F}(b))^2} = 0.\]
  \label{LEM-I2}
\end{lemma}
\begin{proof}
  Observe that $\left(b-S_{n-1}\right)\vee M_{n-1}$ is at least $b/n,$
  and this is achieved with equality when $a_i X_i = b/n$ for every $i
  < n.$ Therefore,
  \begin{align}
    I_2(n,b) &:= \E \left[ Z_1^2(n,b); S_{n-1} > (1-\gamma)b, M_{n-1}
      \leq \gamma b \right] \nonumber\\
    &\leq \bar{F}^2\left(\frac{b}{na_n}\right) \Pr \left\{ S_{n-1} >
      (1-\gamma) b, M_{n-1} \leq \gamma b \right\} \label{INTER-I2}
  \end{align}
  Since $\sum_n na_n < \infty,$ $\sup_n n^2a_n$ exists. Additionally,
  since $\bar{F}(x) = x^{-\alpha}L(x) = x^{-\alpha + o(1)},$ one can
  write
  \[\bar{F}^2\left(\frac{b}{na_n}\right) \leq  (1+o(1)) \left(
    \frac{na_n}{b} \right)^{2(\alpha + o(1))} \leq (1+o(1))
  \left(\sup_n n^2 a_n\right)^{\alpha + o(1)}
  \left(\frac{a_n}{b^2}\right)^{\alpha+o(1)}\] uniformly in $n,$ as $b
  \rightarrow \infty.$ Further, it follows from Proposition
  \ref{PROP-RES-PROB} that for every $n,$
  \[ \Pr \left\{ S_{n-1} > (1-\gamma) b, M_{n-1} \leq \gamma b
  \right\} \leq C_\gamma \bar{F}^{\frac{1-\gamma}{\gamma}}(b),\]
  for some suitable constant $C_\gamma > 0$ and all $b$ large
  enough. Recall the definition of $p_n$ in
  (\ref{PROB-CHOICE-OUTER-RANDOMIZATION}). Since
  $p_n \geq c_b a_nb^{-r},$ it follows from \eqref{INTER-I2} that 
  \begin{align*}
    \frac{I_2(n,b)}{ \left(p_n\bar{F}(b) \right)^2} &\leq C_\gamma
    \left(\sup_n n^2 a_n\right)^{\alpha + o(1)}
    \left(\frac{a_n}{b^2}\right)^{\alpha+o(1)} \frac{b^{2r}}{c_b^2
      a_n^2}\frac{\bar{F}^{\frac{1-\gamma}{\gamma}}(b)}{\bar{F}^2(b)},
  \end{align*}
  uniformly in $n,$ as $b \rightarrow \infty.$ Since $\alpha > 2$ and
  $c_b \sim 1/\sum_n a_n^\alpha$ as $b \rightarrow \infty,$ it follows
  that
  \begin{align*}
    \varlimsup_{b \rightarrow \infty} \sup_{n \geq 1} \frac{I_2(n,b)}{
      \left(p_n\bar{F}(b) \right)^2} = 0
  \end{align*}
  for any choice of $\gamma < 1/3.$
\end{proof}

\noindent Recall that $p_n \geq c_b a_n^\alpha.$ Since $\E
[Z_1^2(n,b)]$ is the sum of $I_1(n,b)$ and $I_2(n,b),$
\begin{align*}
  \frac{\E
    \left[Z_1^2(n,b)\right]}{\left(p_n^{1-\delta}\bar{F}(b)\right)^2}
  \leq
  \frac{I_1(n,b)}{\left(\left(c_ba_n^{\alpha}\right)^{1-\delta}\bar{F}(b)\right)^2} 
  + \frac{I_2(n,b)}{\left(p_n\bar{F}(b)\right)^2}
\end{align*}
for every $n$ and $b.$ Further, we have that $c_b \sim 1/\sum_n
a_n^\alpha$ as $b \rightarrow \infty.$ Then the following is a simple
consequence of Lemmas \ref{LEM-I1} and \ref{LEM-I2}:
\begin{align}
  \label{EFF-Z1}
  \varlimsup_{b \rightarrow \infty} \sup_{n \geq 1} \frac{\E \left[
      Z_1^2(n,b)\right]}{\left( p_n^{1-\delta} \bar{F}(b)\right)^2}
  \leq \lim_{b \rightarrow \infty}\frac{1}{c_b^{2(1-\delta)}} \times
  (1+\delta)^2 + 0 = (1+\delta)^2
  \left(\sum_na_n^\alpha\right)^{2(1-\delta)}.
\end{align}

\subsubsection{Analysis of $Z_2(n,b)$} 
Recall that
\[ Z_2(n,b) = \frac{1}{q(J_n,n)}\left[
  \bar{F}\left(\frac{\xi_1}{a_{_{J_n}}}\right) -
  \bar{F}\left(\frac{\xi_2}{a_{_{J_n}}}\right)\right],\] where
\begin{align*}
  \xi_1 := \left( b-S_n^{(-J_n)}\right) \vee M_n^{(-J_n)} \text{ and }
  \xi_2 := \left( b-S_{n-1}^{(-J_n)}\right) \vee M_n^{(-J_n)}.
\end{align*}
\noindent To upper bound the second moment of $Z_2(n,b),$ we need the
following non-restrictive smoothness assumption on $\bar{F}(\cdot):$
\begin{assumption}
  There exists a $t_0$ such that the slowly varying function
  $L(\cdot)$ in $\bar{F}(x) = L(x)x^{-\alpha}$ is continuously
  differentiable for all $t > t_0.$ Further, $F(\cdot)$ is absolutely
  continuous, the corresponding probability density function
  $f(\cdot)$ is bounded, and there exists a constant $c > 0$ such that
  \begin{align}
    \bar{F}(x) - \bar{F}(y) \leq c (y-x) \frac{\bar{F}(x)}{x}
    \label{TAIL-DIFF-BND}
  \end{align}
    for all $y > x \geq t_0$
  \label{SMOOTHNESS-ASSUMP-II}
\end{assumption}
\noindent One sufficient condition for \eqref{TAIL-DIFF-BND} to hold
is that the slowly varying function $L(\cdot)$ in $\bar{F}(x) =
L(x)x^{-\alpha}$ satisfies
\[ L'(t) = o \left( \frac{L(t)}{t} \right) \quad \text{ as } t
\rightarrow \infty.\]

Similar to the analysis of second moment $Z_1(n,b),$ we upper bound
$\E [Z_2^2(n,b)]$ via the following two terms: Let
\begin{align*}
  J_1(n,b) &:= \E \left[ Z_2^2(n,b); \xi_1 \wedge \xi_2 \geq
    \left(a_{_{J_n}}^\eta \wedge \gamma \right) b\right] \text{ and }\\
  J_2(n,b) &:= \E \left[ Z_2^2(n,b); \xi_1 \wedge \xi_2 < \left(
      a_{_{J_n}}^\eta \wedge \gamma \right) b\right]
\end{align*}
for some fixed $\eta$ and $\gamma$ in $(0,1).$
\begin{lemma}
  Under Assumptions \ref{REG-VAR-TAIL-ASSUMP}, \ref{COEFF-ASSUMP-1}
  and \ref{SMOOTHNESS-ASSUMP-II}, 
  \[ \varlimsup_{b \rightarrow \infty} \sup_{n}
  \frac{J_1(n,b)}{\left(p_n\bar{F}(b)\right)^2} = 0\] for every
  $\gamma$ in $(0,1)$ and some $\eta$ in $(0,1).$
  \label{LEM-J1}
\end{lemma}
\begin{proof}
  Observe that $|\xi_1 - \xi_2| \leq a_n|X_n|.$ Therefore, whenever
  both $\xi_1/a_{_{J_n}}$ and $\xi_2/a_{_{J_n}}$ are larger than
  $t_0,$ due to \eqref{TAIL-DIFF-BND},
  \begin{align*}
    Z_2^2(n,b) \leq \frac{c^2}{q^2(J_n,n)}
    \frac{a_n^2X_n^2}{a_{_{J_n}}^2} \frac{a_{_{J_n}}^2}{\left(\xi_1
        \wedge \xi_2\right)^2} \bar{F}^2\left(\frac{\xi_1 \wedge
        \xi_2}{a_{_{J_n}}} \right).
  \end{align*}
  As a consequence, we have for every $n,$
  \begin{align*}
    J_1(n,b) \leq \E \left[ Z_2^2(n,b); \xi_1 \wedge \xi_2 \geq
      \gamma a_{_{J_n}}^\eta b\right] 
    &\leq c^2a_n^2 \E \left[X_n^2\right] \E \left[
      \frac{1}{q^2(J_n,n)a_{_{J_n}}^2}
      \frac{a_{_{J_n}}^{2(1-\eta)}}{\gamma^2 b^2} \bar{F}^2\left(
        \frac{\gamma b}{a_{_{J_n}}^{1-\eta}} \right) \right].
  \end{align*}
  Then given $\delta > 0,$ for large values of $b,$ due to
  \eqref{LONG-TAIL-EXT},
  \[ \bar{F}\left( \frac{\gamma b}{a_{_{J_n}}^{1-\eta}} \right) \leq
  (1+\delta) \left( \frac{a_{_{J_n}}^{1-\eta}}{\gamma b}
  \right)^{\alpha-\delta}.\] Further, since $q(j,n) =
  a_j/\sum_{i=1}^na_i,$
  \begin{align*}
    J_1(n,b) &\leq (1+\delta)^2
    \frac{c^2a_n^2}{\gamma^{2(\alpha-\delta+1)}}
    \left(\sum_{i=1}^na_i\right)^2\E \left[X_n^2\right] \E \left[
      \frac{a_{_{J_n}}^{\nu}}{b^{2(\alpha-\delta+1)}} \right],
  \end{align*}
  where $\nu := 2(1-\eta)(\alpha-\delta+1)-4.$ If we choose $\eta <
  (\alpha-\delta-1)/(\alpha-\delta+1),$ then $\nu$ is
  positive. Additionally, since $p_n \geq c_b a_nb^{-r}$ (for some $r
  < 1),$
  \begin{align*}
    \sup_{n} \frac{J_1(n,b)}{\left(p_n \bar{F}(b)\right)^2} \leq
    (1+\delta)^2 \frac{c^2}{c_b^2\gamma^{2(\alpha-\delta+1)}}
    \left(\sum_{i=1}^\infty a_i\right)^2 \E\left[X^2\right]
    \frac{b^{-2(\alpha-\delta-r+1)}}{\bar{F}^2(b)}.
  \end{align*} 
  As $\bar{F}(x) \geq (1-\delta)x^{-\alpha-\delta}$ for large values
  of $x,$ it follows that
  \begin{align*}
    \varlimsup_{b \rightarrow \infty} \sup_{n} \frac{J_1(n,b)}{p_n^2
      \bar{F}^2(b)} = 0
  \end{align*}
  for any $\delta$ smaller than $(1-r)/2,$ and this proves the claim.
\end{proof}

\noindent For the analysis of $J_2(n,b),$ we define
\[\kappa := \sup \left\{k: \varlimsup_n n^ka_n < \infty \right\}\] and 
separately analyse the cases $\kappa < \infty$ and $\kappa = \infty.$
If $a_n$ is, for example, polynomially decaying with respect to $n,$
then $\kappa$ happens to be finite. Whereas if $a_n$ is exponentially
decaying with respect to $n,$ then $\kappa$ is infinite. The analysis
for the two cases differ, and are presented below in Lemmas
\ref{LEM-J2-I} and \ref{LEM-J2-II}.
\begin{lemma}
  If $\kappa = \infty,$ then under Assumptions
  \ref{REG-VAR-TAIL-ASSUMP}, \ref{COEFF-ASSUMP-1} and
  \ref{SMOOTHNESS-ASSUMP-II},
  \[ \varlimsup_{b \rightarrow \infty} \sup_{n}
  \frac{J_2(n,b)}{\left(n^{\frac{2}{\eta}}p_n\bar{F}(b)\right)^2} =
  0\] for some $\gamma$ in $(0,1)$ and every $\eta$ in $(0,1).$
  \label{LEM-J2-I}
\end{lemma}
\begin{proof}
  Due to mean value theorem, 
  \begin{align*}
    Z_2(n,b) =
    \frac{1}{q(J_n,n)}\frac{\xi_1-\xi_2}{a_{_{J_n}}}f\left(
      \frac{\zeta}{a_{_{J_n}}}\right)
  \end{align*}
  for some $\zeta$ between $\xi_1$ and $\xi_2.$ Here recall that
  $f(\cdot)$ is the probability density corresponding to the
  distribution $F(\cdot).$ Since $|\xi_1-\xi_2| \leq a_n |X_n|,$ it
  follows from the definition of $J_2(n,b)$ that
  \begin{align*}
    J_2(n,b) \leq \E \left[ \frac{1}{q^2(J_n,n)}
      \frac{a_n^2X_n^2}{a_{_{J_n}}^2}
      f^2\left(\frac{\zeta}{a_{_{J_n}}}\right); \xi_1 \wedge \xi_2 <
      \left(a_{_{J_n}}^\eta \wedge \gamma\right) \right].
  \end{align*}
  Recall that $q(j,n) = a_j/\sum_{i=1}^na_i.$ Then, due to
  H\"{o}lder's inequality,
  \begin{align}
    \frac{J_2(n,b)}{a_n^2} &\leq \left(\sum_{i=1}^na_i\right)^2 \E
    \left[ X_n^{2p}
      f^{2p}\left(\frac{\zeta}{a_{_{J_n}}}\right)\right]^{\frac{1}{p}}
    \E \left[\frac{1}{a_{_{J_n}}^{4q}};\xi_1 \wedge \xi_2 <
      \left(a_{_{J_n}}^\eta \wedge \gamma\right) b
    \right]^{\frac{1}{q}}
    \label{INTER-J2}
  \end{align}
  for some $p,q > 1$ satisfying $p^{-1}+q^{-1}=1$ and $\E[X^{2p}] <
  \infty.$ See that, as in the proof of Lemma \ref{LEM-I2}, $\xi_1
  \wedge \xi_2$ is at least $b/n.$ Therefore,
    \begin{align*}
      \E \left[\frac{1}{a_{_{J_n}}^{4q}};\xi_1 \wedge \xi_2 <
        \left(a_{_{J_n}}^\eta \wedge \gamma\right) b \right] &= \E
      \left[\left(\frac{b}{\xi_1 \wedge
            \xi_2}\right)^{\frac{4q}{\eta}};\xi_1 \wedge \xi_2 <
        \left(a_{_{J_n}}^\eta \wedge \gamma\right) b\right]\\
      &\leq n^{\frac{4q}{\eta}} \Pr \left\{ \xi_1 \wedge \xi_2 <
        \left(a_{_{J_n}}^\eta \wedge \gamma\right) b \right\}.
    \end{align*}
    From the definition of $\xi_1$ and $\xi_2,$ it is immediate that
    for every $n,$
   \begin{align}
     \Pr \left\{ \xi_1 \wedge \xi_2 < \left(a_{_{J_n}}^\eta \wedge
         \gamma\right) b \left|\frac{}{}\right. J_n \right\} &\leq \Pr
     \left\{ S_n^{(-J_n)} \vee S_{n-1}^{(-J_n)} > (1-\gamma)b,
       M_n^{(-J_n)} \leq \gamma b \left|\frac{}{}\right. J_n \right\}
     \nonumber\\
     &\leq c_\gamma \bar{F}^{\frac{1-\gamma}{\gamma}}(b)
     \label{INTER-J2-RES-PROB}
    \end{align}
    for some constant $c_\gamma$ and all $b$ large enough, because of
    union bound and Proposition \ref{PROP-RES-PROB}. Further, recall
    that $p_n \geq c_b a_n b^{-r}, \E[X^{2p}]$ is finite, and
    $f(\cdot)$ is bounded. These observations, in conjunction with
    \eqref{INTER-J2}, result in
    \begin{align*}
      \sup_{n \geq
        1}\frac{J_2(n,b)}{\left(n^{\frac{2}{\eta}}p_n\bar{F}(b)\right)^2}
      = O \left( \frac{b^{2r}\bar{F}^{\frac{1-\gamma}{\gamma
              q}}(b)}{\bar{F}^2(b)}\right), \text{ as } b \rightarrow
      \infty.
    \end{align*}
    Given $r < 1$ and $q,$ one can choose $\gamma$ suitably so that
    $b^{2r}\bar{F}^{\frac{1-\gamma}{\gamma q}}(b)$ vanishes as $b
    \rightarrow \infty.$ This proves the claim.
\end{proof}
\begin{lemma}
  If $\kappa < \infty,$ then under Assumptions
  \ref{REG-VAR-TAIL-ASSUMP}, \ref{COEFF-ASSUMP-1} and
  \ref{SMOOTHNESS-ASSUMP-II},
  \[ \varlimsup_{b \rightarrow \infty} \sup_{n}
  \frac{J_2(n,b)}{\left(p_n\bar{F}(b)\right)^2} = 0\] for some
  $\gamma$ in $(0,1)$ and every $\eta$ in $(0,1).$
  \label{LEM-J2-II}
\end{lemma}
\begin{proof}
  Observe that the argument leading to \eqref{INTER-J2} in the proof
  of Lemma \ref{LEM-J2-I} holds irrespective of whether $\kappa$ is
  finite or not. To proceed further, see that
  \begin{align}
    \E \left[\frac{1}{a_{_{J_n}}^{4q}};\xi_1 \wedge \xi_2 <
      \left(a_{_{J_n}}^\eta \wedge \gamma\right) b \right] = \E \left[
      \frac{1}{a_{_{J_n}}^{4q}} \Pr \left\{ \xi_1 \wedge \xi_2 <
        \left(a_{_{J_n}}^\eta \wedge \gamma\right) b \
        \left|\frac{}{}\right. \ J_n \right\} \right].
    \label{J2-INTER-II}
  \end{align}
  It follows from the definition of $\xi_1$ and $\xi_2$ that
  \begin{align*}
    \Pr \left\{ \xi_1 \wedge \xi_2 < \left(a_{_{J_n}}^\eta \wedge
        \gamma\right) b \ \left|\frac{}{}\right. \ J_n \right\} &\leq
    \Pr \left\{ M_n^{(-J_n)} < a_{_{J_n}}^\eta b \
      \left|\frac{}{}\right. \ J_n \right\} = \prod_{i=1, i \neq
      J_n}^n F\left( \frac{ a_{_{J_n}}^\eta b}{a_i} \right).
  \end{align*}
  For any fixed $k > \kappa,$ there exists a positive constant
  $\tilde{c}_k$ such that $n^ka_n \geq \tilde{c}_k$ for all $n.$ Then
  \begin{align*}
    \Pr \left\{ \xi_1 \wedge \xi_2 < \left(a_{_{J_n}}^\eta \wedge
        \gamma\right) b \ \left|\frac{}{}\right. \ J_n \right\} &\leq
    \prod_{i=1, i \neq J_n}^n F\left( \frac{ i^ka_{_{J_n}}^\eta
        b}{\tilde{c}_k} \right) \leq F(1)^{
      \left(\frac{\tilde{c}_k}{a_{_{J_n}}^\eta
          b}\right)^{\frac{1}{k}}-2},
  \end{align*}
  where we have simply excluded the last $n-\lceil
  (\tilde{c}_k/(a_{_{J_n}}^\eta b))^{1/k} \rceil$ terms in the product
  to get an upper bound. This inequality, along with
  \eqref{INTER-J2-RES-PROB}, results in the following loose bound
  which is enough for our purposes:
  \begin{align*}
    \Pr \left\{ \xi_1 \wedge \xi_2 < \left(a_{_{J_n}}^\eta \wedge
        \gamma\right) b \ \left|\frac{}{}\right. \ J_n \right\} \leq c
    F(1)^{{\frac{1}{2} \left( \frac{\tilde{c}_k}{a_{_{J_n}}^\eta
            b}\right)^{\frac{1}{k}}}-1}
    \bar{F}^{\frac{1-\gamma}{2\gamma}}(b),
  \end{align*}
  for some constant $c > 0.$ Using this in \eqref{J2-INTER-II}, we
  have that
  \begin{align*}
    \E \left[\frac{1}{a_{_{J_n}}^{4q}};\xi_1 \wedge \xi_2 <
      \left(a_{_{J_n}}^\eta \wedge \gamma\right) b \right] \leq c
    b^{\frac{4q}{\eta}} \E \left[ \frac{1}{\left(a_{_{J_n}}^\eta b
        \right)^{\frac{4q}{\eta}}} F(1)^{{\frac{1}{2} \left(
            \frac{\tilde{c}_k}{a_{_{J_n}}^\eta
              b}\right)^{\frac{1}{k}}}-1}\right]
    \bar{F}^{\frac{1-\gamma}{2\gamma}}(b)
  \end{align*}
  Since $x^{\frac{4qk}{\eta}}F(1)^{x-1}$ is bounded for positive
  values of $x,$ the expectation term in the right hand side of the
  above equation is finite. Further, $p_n \geq c_b a_n b^{-r}$. As a
  consequence, we have from \eqref{INTER-J2} that
  \begin{align*}
    \sup_{n \geq 1}\frac{J_2(n,b)}{\left(p_n\bar{F}(b)\right)^2} =
    O\left( b^{\frac{4}{\eta} + 2r}
      \frac{\bar{F}^{\frac{1-\gamma}{2q\gamma}}(b)}{\bar{F}^2(b)}
    \right),
  \end{align*}
  which, for suitably chosen $\gamma,$ vanishes to $0$ as $b \rightarrow
  \infty.$ This concludes the proof.
\end{proof}

\noindent Since $\E [Z_2^2(n,b)]$ is the sum of $J_1(n,b)$ and
$J_2(n,b),$ when $\kappa = \infty,$ due to Lemmas \ref{LEM-J1} and
\ref{LEM-J2-I}, one can choose $\eta$ and $\gamma$ in $(0,1)$ such
that
\begin{align}
  \label{EFF-Z2-I}
  \varlimsup_{b \rightarrow \infty} \sup_{n \geq 1} \frac{\E \left[
      Z_2^2(n,b)\right]}{\left( n^{\frac{2}{\eta}}
      p_n\bar{F}(b)\right)^2} = 0.
\end{align}
Similarly, when $\kappa < \infty,$ due to Lemmas \ref{LEM-J1} and
\ref{LEM-J2-II},
\begin{align}
  \label{EFF-Z2-II}
  \varlimsup_{b \rightarrow \infty} \sup_{n \geq 1} \frac{\E \left[
      Z_2^2(n,b)\right]}{\left(p_n\bar{F}(b)\right)^2} = 0.
\end{align}

\subsection{Proof of Theorem \ref{THM-EFF-OVRALL}} 
Recall that
\[ Z(b) = \frac{\Zl(N,b)}{p_{_N}} = \frac{Z_1(N,b) +
  Z_2(N,b)}{p_{_N}}.\] 
Therefore,
\begin{align*}
\frac{\E\left[Z^2(b)\right]}{\bar{F}^2(b)} =
  \E\left[\frac{Z_1^2(N,b)}{p_{_N}^2\bar{F}^2(b)}\right]
  + \E\left[ \frac{Z_2^2(N,b)}{p_{_N}^2\bar{F}^2(b)}\right] +
\E\left[\frac{Z_1(N,b)}{p_{_N}\bar{F}(b)} \right]
  \E\left[\frac{Z_2(N,b)}{p_{_N}\bar{F}(b)} \right].
\end{align*}
Then due to Jensen's inequality,
\begin{align}
  \frac{\E\left[Z^2(b)\right]}{ \bar{F}^2(b)} \leq
  \E\left[ \frac{Z_1^2(N,b)}{p_{_N}^2\bar{F}^2(b)}\right]
  + \E\left[\frac{Z_2^2(N,b)}{p_{_N}^2\bar{F}^2(b)}\right] +
  \sqrt{\E\left[ \frac{Z_1^2(N,b)}{p_{_N}^2\bar{F}^2(b)}\right]}
  \sqrt{\E\left[ \frac{Z_2^2(N,b)}{p_{_N}^2\bar{F}^2(b)}\right]}.
  \label{INTER}
\end{align}
Now consider, for example, the first term in the right hand side of
the above inequality. Due to the uniform convergence result on $\E
[Z_1^2(n,b)]$ in \eqref{EFF-Z1}, there exists a constant $c_1$ such
that 
\begin{align*}
  \frac{\E\left[Z_1^2(N,b)\ |\
      N\right]}{\left(p_{_N}\bar{F}(b)\right)^2} \leq
  c_1(1+\delta)^2p_{_N}^{-2\delta} \left( \sum_n
    a_n^\alpha\right)^{2(1-\delta)}
\end{align*}
for every $\delta$ and $b.$ Since $\sum_n na_n$ exists, $\E
p_{_N}^{-2\delta} < \infty$ for all $\delta$ small enough. As $\delta$
can be arbitrarily small, due to reverse Fatou's lemma, it follows
from \eqref{EFF-Z1} that
\begin{align}
\label{TERM-1}
\varlimsup_{b \rightarrow \infty}
\E\left[\frac{Z_1^2(N,b)}{p_{_N}^2\bar{F}^2(b)}\right] \leq 
\E \left[ \varlimsup_{b \rightarrow \infty} \frac{\E\left[Z_1^2(N,b)\
      |\ N\right]}{\left(p_{_N}\bar{F}(b)\right)^2} \right] \leq
\left(\sum_n a_n^\alpha \right)^2.
\end{align}
Similarly, one can conclude from \eqref{EFF-Z2-I} and
\eqref{EFF-Z2-II} that for every $b,$
\begin{align*}
  \frac{\E\left[Z_2^2(N,b)\ |\
      N\right]}{\left(p_{_N}\bar{F}(b)\right)^2} \leq 
\begin{cases}
  c_2 N^{\frac{4}{\eta}} & \text{if } \kappa = \infty \\
  c_2 & \text{if } \kappa < \infty.
\end{cases}
\end{align*}
for some constant $c_2$. Observe that $\E N^{\frac{4}{\eta}} < \infty$
for any fixed $\eta$ because when $\kappa = \infty,$ $p_n$ is
exponentially decaying with respect to $n.$ Then as a consequence of
\eqref{EFF-Z2-I} and \eqref{EFF-Z2-II}, due to dominated convergence,
\begin{align*}
  \lim_{b \rightarrow \infty}
  \E\left[ \frac{Z_2^2(N,b)}{p_{_N}^2\bar{F}^2(b)} \right] =
  \E\left[ \lim_{b \rightarrow \infty} \frac{\E\left[Z_2^2(N,b)\ |\
        N\right]}{\left(p_{_N}\bar{F}(b)\right)^2}\right] = 0.
\end{align*}
This conclusion, along with \eqref{INTER} and \eqref{TERM-1}, results
in
\begin{align*}
  \lim_{b \rightarrow \infty}
  \frac{\E\left[Z^2(b)\right]}{\bar{F}^2(b)} \leq
  \left(\sum_n a_n^\alpha\right)^2.
\end{align*}
Further, $\Pr \{ S > b\} \sim \sum_n a_n^\alpha \bar{F}(b)$ as $b
\rightarrow \infty.$ Therefore,
\begin{align*}
  \varlimsup_{b \rightarrow \infty} \frac{\E\left[Z^2(b)\right]}{\Pr\{
    S > b\}^2} \leq 1.
\end{align*}
Additionally, since $Z(b)$ is an unbiased estimator of $ \Pr\{S >
b\},$ $\E[Z^2(b)]$ must be larger than $\Pr \{ S > b\}^2$ because of
Jensen's inequality.  This proves the theorem. \hfill{$\Box$}

\subsection{A note on computational complexity of the simulation
  procedure}  
\label{SEC-COMP-CPXTY-CHAP5}
Given $b > 0,$ our objective has been to devise an algorithm that
returns a number in the interval
$((1-\epsilon)\Pr\{S>b\},(1+\epsilon)\Pr\{S>b\})$ with probability at
least $1-\delta.$ In Section \ref{SEC-SIM-METH-CHAP5}, we proposed to
take average of values returned by several runs of Algorithm
\ref{ALGO-OVRALL} as the estimate of $\Pr\{S > b\}.$ Assuming that
tasks like performing basic arithmetic operations, generating uniform
random numbers, evaluating $F(x)$ at specified $x$, all require unit
computational effort, it is immediate that each call to the procedure
\textsc{LocalSimulation}$(n,b)$ expends at most $C n$ computational
effort, for some positive constant $C,$ irrespective of the value of
$b.$ Given $b > 0,$ if one makes $N_b$ calls to Algorithm
\ref{ALGO-OVRALL} and returns the average of returned values of $Z(b)$
as the overall estimate, then
\begin{itemize}
\item[1)] the estimate lies within the desired interval with
  probability at least $\epsilon^{-2} \textnormal{CV}^2[Z(b)]/N_b,$
  where $\textnormal{CV}[Z_b] = \textnormal{Var}[Z_b]/\E[Z_b]^2$ is
  the coefficient of variation of $Z_b,$ and
\item[2)] the overall computational effort is at most $C N N_b,$ where
  $N$ is the auxiliary random variable drawn according to the
  probability mass function $(p_n: n \geq 1)$ in Algorithm
  \ref{ALGO-OVRALL}.
\end{itemize}
Due to Theorem \ref{THM-EFF-OVRALL}, we have that
$\textnormal{CV}[Z(b)] = o(1),$ as $b \rightarrow \infty.$ Therefore, it
is enough to choose $N_b = c \epsilon^{-2}\delta^{-1}$ for some
positive constant $c.$ Further, note that
\begin{align*}
  \E [N] = \sum_n n p_n = c_b \sum_n n \left( a_n^\alpha +
    \frac{a_n}{b^r}\right).
\end{align*}
First, observe that $\sum_n a_n < \infty$ because of Assumption
\ref{COEFF-ASSUMP-1}. Additionally, since $c_b \sim \sum_n a_n^\alpha$
as $b \rightarrow \infty,$ we have $\E N = O(1)$ as $b \rightarrow \infty.$
Therefore, the overall computational effort is just $O(1)$ as $b
\rightarrow \infty.$ Thus, despite the difficulties that the definition
of $S$ involves infinitely many random variables and $\Pr\{S > b\}$ is
arbitrarily small for large values of $b,$ our work establishes that
one can compute $\Pr\{S > b\}$ without any bias by expending only a
computational effort that is uniformly bounded in $b.$

\section{A numerical example}
\label{SEC-NUM-EG-CHAP3}
In this section, we present the results of a numerical simulation
experiment that demonstrates the efficiency of our estimator. Take
$(X_n: n \geq 1)$ to be iid copies of a Pareto random variable $X$
satisfying $\Pr\{ X > x\} = 1 \wedge x^{-4}.$ Additionally, take
$a_n = 0.9^n$ and let $S = \sum_n a_n X_n.$ We use $N = 10,000$
simulation runs to estimate $\Pr \{ S > b \}$ for various values of
$b$ listed in Table \ref{TAB-NUM-RES-CHAP5}. The parameter $r$ in the
choice of probabilities $p_n$ in the expression
\ref{PROB-CHOICE-OUTER-RANDOMIZATION} is taken to be 1. The values
listed in Column 3 correspond to the estimate obtained from 10,000
runs of our simulation algorithm. It is instructive to compare the
simulation estimates in Column 3 with the crude asymptotic
$\bar{F}(b) \sum_n a_n^\alpha$ listed in Column 2. The empirically
observed coefficient of variation of our simulation estimators is
listed in Column 5. Although it is required in the proof of Theorem
\ref{THM-EFF-OVRALL} that $r > 1,$ it can be inferred from Column 5
that the choice $r = 1$ yields estimators that have coefficient of
variation that decreases to 0 as $b$ is increased.

\begin{table}[htb!]
  \centering
  \caption{Numerical result for the simulation of $\Pr\{ S > b\}$- here
    CV denotes the empirically observed coefficient of variation based
    on 10,000 simulation runs}
  \begin{tabular}{l c p{2.1cm} p{2.5cm} p{1cm}}
    \hline
    b    & Asymptotic $\bar{F}(b) \sum_n a_n^\alpha$ & Estimate for $\Pr\{S>b\}$&  Standard Error    & CV\\ \hline
    200 & 1.19 $\times 10^{-9}$ & 1.49$\times 10^{-9}$ & 1.61 $\times 10^{-11}$ &1.08\\ 
    500 & 3.05 $\times 10^{-11}$ & 3.32$\times 10^{-11}$ & 1.54 $\times 10^{-13}$& 0.47\\
  1000 & 1.91 $\times 10^{-12}$ & 1.97$\times 10^{-12}$ & 8.43 $\times 10^{-15}$ &0.42\\\hline
  \end{tabular}
  \label{TAB-NUM-RES-CHAP5}
\end{table}

\bibliography{Dissertation}
\bibliographystyle{abbrv}

\section*{Appendix}
\noindent
We present proof of Proposition \ref{PROP-RES-PROB} here in the
appendix.  To accomplish this 
we need Lemma \ref{Rthx} first, which is stated and proved below.

\begin{lemma}
  For any pair of sequences $\{x_n\},\{ \fin \}$ satisfying $x_n
  \rightarrow \infty$ and $\phi _n x_n \rightarrow \infty,$ the integral,
  \begin{align*}
    \int_{-\infty}^{x_n} e^{\fin x}F(dx) \leq 1 + c \fin^\kappa + e^{2
      \alpha }\bar{F}\left( \frac{2 \alpha}{\fin} \right) + e^{\fin
      x_n} \bar{F}(x_n) (1+ o(1)),
  \end{align*}
  as $n \rightarrow \infty,$ for any $1 < \kappa < \alpha \wedge 2,$ and
  some constant $c$ which does not depend on $n$ and $b.$
  \label{Rthx}
\end{lemma}
\begin{proof}
  We split the region of integration into $(-\infty, \gn]$ and $(\gn,
  x_n]$ for some constant $\gamma > 0$; the partition is such that the
  integrand stays bounded in the former region.\\
  Let $I_1 := \int_{-\infty}^{\gn} e^{\fin x} F(dx)$ and $I_2 :=
  \int_{\gn}^{x_n} e^{\fin x} F(dx).$\\
  For any $\kappa \in (1,2]$ and $y > 0,$ it is easily verified that
  \[ e^x \leq 1 + x + |x|^\kappa e^y, \quad x \in (-\infty,y].\]
  Therefore,
\begin{align}
  I_1 &\leq \int_{-\infty}^{\gn}\left(1 + \fin x + \fin^\kappa
    |x|^\kappa \exp(\fin \cdot \gn) \right) F(dx)
  \nonumber\\
  &\leq \int_{-\infty}^{\gn}F(dx) + \fin\int_{-\infty}^{\gn}xF(dx) +
  \fin ^\kappa \ega\int_{-\infty}^{\gn} |x|^\kappa F(dx)
  \nonumber\\
  &\leq \int_{-\infty}^{\infty}F(dx) + \fin
  \int_{-\infty}^{\infty}xF(dx) + \fin ^\kappa \ega
  \int_{-\infty}^{\infty} |x|^\kappa F(dx) \nonumber\\
  &= 1 + c \fin ^\kappa, \label{int-I1}
\end{align}
where $c := \ega \int_{-\infty}^{\infty} |x|^\kappa F(dx) < \infty$
because $\E |X|^\kappa < \infty;$ this follows because
$\kappa <
\alpha.$ 
We have also used $\E X = 0$ to arrive at \eqref{int-I1}.  Integrating
by parts for the second integral $I_2:$
\begin{align}
  I_2 &= -\int_{\gn}^{x_n} e^{\fin x} \bar{F}(dx) \nonumber\\
  &= e^{\fin \gn} \bar{F}\left(\frac{\gamma}{\phi_n}\right) - e^{\fin
    x_n}\bar{F}(x_n) + \fin
  \int_{\gn}^{x_n} e^{\fin x} \bar{F}(x)dx \nonumber\\
  &\leq \ega \bar{F}\left(\frac{\gamma}{\phi_n} \right) + I_2',
\label{int-I2}
\end{align}
where, $I_2' := \fin \int_{\gn}^{x_n} e^{\fin x}
\bar{F}(x) dx.$ Now the change of variable $u = \fin (x_n - x)$  results in:
\begin{align}
  I_2' &= e^{\fin x_n} \int_0^{\fin x_n - \gamma} e^{-u}\bar{F} \left(x_n -
    \frac{u}{\fin } \right) du \nonumber\\
  &= e^{\fin x_n} \bar{F}(x_n) \int_0^{\fin x_n - \gamma} e^{-u} g_n(u)
  du,
\label{int-I2'}
\end{align}
where,
\begin{equation*}
  g_n(u) := \frac{\bar{F} \left(x_n - \frac{u}{\fin }
    \right)}{\bar{F}(x_n)} = \frac{\bar{F} \left(x_n \left(1-
        \frac{u}{\fin x_n} \right) \right)}{\bar{F}(x_n)}.
\end{equation*}
Since $L(\cdot)$ is slowly varying and $\fin x_n \rightarrow \infty,$
given any $\delta > 0,$ it follows from \eqref{LONG-TAIL-EXT} that, 
\begin{equation*}
(1- \delta) \left(1-\frac{u}{\fin x_n} \right) ^{-\alpha + \delta}
\leq g_n(u) \leq (1 + \delta )\left(1-\frac{u}{\fin x_n} \right) ^{-\alpha - \delta}.
\end{equation*}
for all $n$ large enough. So for any fixed $u,$ we have $g_n(u)
\rightarrow 1$ as $n \rightarrow \infty.$ Now fix $\delta =
\frac{\alpha}{2}.$ Then for $n$ large enough,
\begin{equation}
  g_n(u) \leq \left( 1 + \frac{\alpha}{2} \right) \left( 1
    -\frac{u}{\fin x_n }  \right) ^{-\frac{3\alpha }{2}}.
\label{gn-bound}
\end{equation}
Let $h(u) = \left( 1 -u/\fin x_n \right) ^{-\frac{3\alpha
  }{2}}$. Since $\log h(0) = 0$ and $\frac{d}{du}\left( \log(h(u)
\right) \leq \frac{3 \alpha}{2 \gamma }$ for $0 \leq u \leq \fin x_n -
\gamma ,$ we have $h(u) \leq \exp({3 \alpha u}/{2 \gamma })$ on the
same interval. Therefore if we choose $\gamma = 2 \alpha,$ the
integrand in $I_2'$ is bounded for large enough $n$ by an integrable
function as below:
\begin{align*}
  \left| e^{-u}g_n(u)\mathbf{1}(0 \leq u \leq \fin x_n - \gamma)
  \right| &\leq \left| e^{-u} \left( 1 + \frac{\alpha }{2} \right)
    h(u)\mathbf{1}(0 \leq u \leq \fin x_n - \gamma) \right|\\
  &\leq \left( 1 + \frac{\alpha }{2} \right) e^{-u + \frac{3 \alpha
      u}{2 \gamma }} = \left( 1 + \frac{\alpha }{2} \right) e
  ^{-\frac{u}{4}}.
\end{align*}
Applying dominated convergence theorem, we get
\begin{equation*}
\int_0^{\fin x_n -
  \gamma} e^{-u} g_n(u)du \sim 1 \text{ as } n \rightarrow \infty.
\end{equation*}
Since $\int_{-\infty}^{x_n} e^{\fin x}F(dx) = I_1 + I_2,$ combining
this result with \eqref{int-I1}, \eqref{int-I2} and \eqref{int-I2'},
completes the proof.
\end{proof}

\begin{proof}[Proof of Proposition \ref{PROP-RES-PROB}]
  Observe that for any $n$ and $j,$
  \[ \left\{ M_n^{(-j)} \leq \frac{b}{k}\right\} = \bigcap_{i=1, i
    \neq j}^n \left\{ X_i \leq \frac{b}{k a_i}\right\}.\] Then for any
  $\theta > 0,$
  \begin{align*}
    \Pr \left\{ S_n^{(-j)} > b, M_n^{(-j)} \leq \frac{b}{k}\right\}
    \leq \exp(-\theta b) \prod_{i=1,i \neq j}^n \E \left[ \exp(\theta
      a_i X_i); X_i \leq \frac{b}{ka_i} \right]
  \end{align*}
  because of a simple application of Markov's inequality. If $\theta$
  is chosen such that $\theta b \rightarrow \infty$ as $b \rightarrow
  \infty,$ from Lemma \ref{Rthx}, we have
  \begin{align*}
  \E \left[ \exp(\theta a_i X_i); X_i \leq \frac{b}{ka_i} \right] \leq
  1 + c \theta^2 a_i^2 + e^{2 \alpha }\bar{F}\left( \frac{2
      \alpha}{\theta a_i} \right) + \exp \left(\theta \frac{b}{k}
  \right) \bar{F}\left( \frac{b}{ka_i}\right) (1+ o(1)),
\end{align*}
uniformly in $i,$ as $b \rightarrow \infty.$ Since $1 + x \leq \exp(x),$
\begin{align}
  \Pr & \left\{ S_n^{(-j)} > b, M_n^{(-j)} \leq \frac{b}{k}\right\}
  \nonumber\\
  &\quad \leq \exp(-\theta b) \prod_{i=1,i \neq j}^n \exp \left( c
    \theta^2 a_i^2 + e^{2 \alpha }\bar{F}\left( \frac{2 \alpha}{\theta
        a_i} \right) + \exp \left(\theta \frac{b}{k} \right)
    \bar{F}\left( \frac{b}{ka_i}\right) (1+ o(1))  \right) \nonumber\\
  & \quad \leq \exp \left( -\theta b + c \theta^2 \sum_i a_i^2 + e^{2
      \alpha } \sum_i \bar{F}\left( \frac{2 \alpha}{\theta a_i}
    \right) + \bar{F}\left( \frac{b}{k}\right) \exp \left(\theta
      \frac{b}{k} \right) \sum_i a_i^{\alpha-\epsilon} (1+o(1))
  \right), \label{INTER-THETA}
\end{align}
for any given $\epsilon > 0,$ due to \eqref{LONG-TAIL-EXT}, uniformly
in $j$ and $n,$ as $b \rightarrow \infty.$ Observe that
\[ \theta_b := -\frac{k}{b}\log \left( \frac{\sum_i a_i^\alpha}{k}
  \bar{F} \left( \frac{b}{k}\right)\right)\] is the minimizer of
$-\theta b + \sum_i a_i^\alpha \bar{F}(b/k) \exp(\theta b/k),$ and it
approximately minimizes the right hand side of
\eqref{INTER-THETA}. Since $\theta_b \searrow 0$ and $\sum_i
a_i^{\alpha-\epsilon} < \infty$ for small enough $\epsilon,$ it
follows from \eqref{LONG-TAIL-EXT} that
\begin{align*}
  \sum_i \bar{F} \left( \frac{2\alpha}{\theta_b a_i}\right) \leq
  (1+\epsilon)\sum_i
  \left(\frac{a_i}{2\alpha}\right)^{\alpha-\epsilon} \bar{F}\left(
    \frac{1}{\theta_b}\right)  = o( \theta_b),\\
  \theta_b^2 = o \left( \theta_b \right), \text{ and } \bar{F}\left(
    \frac{b}{k}\right) \exp \left(\theta_b \frac{b}{k} \right) \sum_i
  a_i^\alpha = k,
\end{align*}
as $b \rightarrow \infty.$ Therefore, uniformly for every $n$ and $j \leq
n,$
\begin{align*}
  \Pr \left\{ S_n^{(-j)} > b, M_n^{(-j)} \leq \frac{b}{k}\right\}
  &\leq \exp \left( k\log \left( \frac{\sum_i a_i^\alpha}{k} \bar{F}
      \left( \frac{b}{k}\right)\right) + o(1) + k (1+o(1)) \right)\\
  &= \exp(k+o(1)) \left( \frac{\sum_i a_i^\alpha}{k}
    \bar{F}\left(\frac{b}{k}\right)\right)^k,
\end{align*}
as $b \rightarrow \infty.$ This proves the claim.
\end{proof}

\end{document}